\documentclass[11pt, reqno]{amsart}
\usepackage{mathtools, amssymb}
\usepackage[utf8]{inputenc}
\usepackage{tikz}
\usepackage{pgfplots}
\pgfplotsset{compat=1.18}
\usepackage{mathrsfs}
\usepgfplotslibrary{colormaps}
\usepgfplotslibrary{fillbetween}
\usepackage{mathptmx}
\usepackage{xcolor}
\usepackage[colorlinks = true,
            linkcolor = black,
            urlcolor  = blue,
            citecolor = blue,
            anchorcolor = blue]{hyperref}
\newtheorem{theorem}{Theorem}[section]

\newtheorem{proposition}[theorem]{Proposition}
\newtheorem{corollary}[theorem]{Corollary}

\theoremstyle{definition}

\newtheorem{definition}[theorem]{Definition}
\newtheorem{remark}[theorem]{Remark}
\newtheorem{example}[theorem]{Example}

\newtheorem*{theorem*}{Theorem}
\newtheorem*{question*}{Open question}

\newcommand{\nperp}{\perp_{\delta}^\varepsilon}
\newcommand{\zperp}{\perp_B^{\varepsilon}}

\title[Norming Approximate Orthogonality]
{Norming Approximate Orthogonality in Normed Linear Spaces}

\author{Saikat Roy}

\address[Roy]{Department of Mathematics, School of Advanced Sciences, VIT-AP University, Beside AP Secretariat, Amaravati, 522241, Andhra Pradesh, India.}
\email{saikatroy.cu@gmail.com, saikat.roy@vitap.ac.in}


\begin{document}

\begin{abstract}
We introduce and study the notion of \emph{norming approximate orthogonality}, a two-parameter generalization of Birkhoff--James orthogonality in normed linear spaces. For 
$\delta, \varepsilon \in [0,1)$ with $\varepsilon < (1-\delta)^2$, we say 
$x \nperp y$ in $X$ if there exists $f \in X^*$ with 
$|f(x)| \geq (1-\delta)\|f\|\|x\|$ and $|f(y)| \leq \frac{\varepsilon}{1-\delta}\|f\|\|y\|$, 
simultaneously relaxing both the norming condition on $x$ and the vanishing 
condition on $y$. It is proved that
\[
x\nperp y \iff \|x+\lambda y\|\geq (1-\delta)\|x\|-\frac{\varepsilon}{1-\delta}\|\lambda y\|~\qquad \forall ~\text{scalars}~\lambda.
\]
This framework interpolates between two notions of approximate orthogonality in normed linear spaces due to Chmieli\'nski and Dragomir, and recovers three existing notions of orthogonality in extreme cases: exact Birkhoff--James orthogonality at $\delta = \varepsilon = 0$, the approximate 
orthogonality of Chmieli\'nski at $\delta = 0$, and the approximate orthogonality 
of Dragomir at $\varepsilon = 0$. In Hilbert spaces such an orthogonality is characterized by an inner product bound $L(\delta,\varepsilon)<1$, namely
\[
x\nperp y \iff |\langle x, y \rangle| \leq L(\delta,\varepsilon)\|x\|\|y\|,
\qquad L(\delta,\varepsilon)=\cos(\theta_\varepsilon-\theta_\delta),
\]
where $\cos\theta_\delta = 1-\delta$ and
$\cos\theta_\varepsilon = \frac{\varepsilon}{1-\delta}$.
A two-parameter proximity result generalizing Chmieli\'nski's characterization of $\perp_B^\varepsilon$ is established. A dual formulation of norming approximate orthogonality is established with a complete equivalence in the reflexive case. We apply our results to (vector-valued) continuous function spaces, which extends some earlier results and recovers  a few operator-theoretic results with alternative proofs using measure theoretic techniques.
\end{abstract}

\maketitle

{\small \tableofcontents}

\bigskip

\section{Introduction}

\noindent All normed linear spaces are over $\mathbb{C}$; results remain valid over 
$\mathbb{R}$ with straightforward modifications. We write $X^*$ for the (topological) dual of a normed linear space $X$, and $B_X$, $S_X$ for the closed unit ball and unit sphere of $X$, respectively. For any $\epsilon > 0$, $B(0, \epsilon)$ denotes the closed disc of radius $\epsilon$ centred at $0$ (origin) in the complex plane. For any non-zero $x$ in a normed linear space $X$, we denote the collection of all support functionals at $x$, by $J(x)=\{f\in S_{X^*}:~f(x)=\|x\|\}$, which is non-empty by the Hahn-Banach Theorem.

\subsection{Motivation and Background}

\noindent Orthogonality in a normed linear space is captured by the Birkhoff--James 
relation~\cite{Bir}: $x \perp_B y$ if $\|x + \lambda y\| \geq \|x\|$ 
for all scalars $\lambda$, or equivalently, if there exists $f \in X^*$ with 
$\|f\| = 1$, $f(x) = \|x\|$ and $f(y) = 0$. Birkhoff-James orthogonality has a deep connection with the norm attainment set of functionals \cite{Jam1, Jam2, roy}, namely $x\perp_B y$ if and only if there exists a continuous linear functional that attains 
its norm at $x$ and vanishes at $y$. Two independent lines of 
relaxation have been pursued in the literature. Dragomir~\cite{Drag} introduced 
the first approximate version, defining $x \perp^{\varepsilon}_{D} y$ for 
$\varepsilon \in [0,1)$ by the following norm inequality
\[
\|x + \lambda y\| \geq (1-\varepsilon)\|x\| \quad \text{for all scalars } \lambda,
\]
Another notion of approximate orthogonality is introduced by Chmieli\'nski~\cite{Chm1} defining $x \perp_B^\varepsilon y$ by
\[
\|x + \lambda y\|^2 \geq \|x\|^2 - 2\varepsilon\,\|x\|\,\|\lambda y\| 
\quad \text{for all scalars } \lambda,
\]
and this relation has since been studied 
extensively~\cite{Chm2, Chm3, Chm4, ChmGrWoj, ChmStWoj, ChmWoj, woj, Woj1, Pal, Alt}. 
While $x \perp_B^\varepsilon y$ possesses an equivalent form in terms of continuous linear functionals attaining norm at $x$ and approximately vanishes at $y$ \cite{ChmStWoj, woj} 
\[
x \perp_B^\varepsilon y \iff \exists\, f \in X^* : \|f\| = 1,\ 
f(x) = \|x\|,\ |f(y)| \leq \varepsilon\|y\|,
\]
a functional description of $x \perp^{\varepsilon}_{D} y$ was obtained in
\cite[Lemma 3.2]{MSP}. However, a common parametric framework interpolating between the two relaxations and reducing to Birkhoff--James orthogonality in the limiting case $\delta=\varepsilon=0$ remained absent. The present paper addresses both simultaneously, asking what orthogonality is induced when one requires a functional that 
\emph{approximately} attains its norm at $x$ and \emph{approximately} 
vanishes at $y$.

\begin{definition}\label{1st Defn}
Let $X$ be a normed linear space and $x, y \in X$. Let $\delta, \varepsilon \in [0,1)$ 
satisfy $\varepsilon < (1-\delta)^2$. Then $x$ is said to be 
\emph{$(\delta,\varepsilon)$-norming approximately orthogonal} to $y$, written 
$x \nperp y$, if
\[
\exists\, f \in X^* \setminus\{0\} : \quad |f(x)| \geq (1-\delta)\|f\|\|x\|, 
\quad |f(y)| \leq \frac{\varepsilon}{1-\delta}\|f\|\|y\|.
\]
 Since both conditions are homogeneous of degree one in $f$, replacing $f$ by 
$f/\|f\|$ shows that one may equivalently take $f\in S_{X^*}$, and we shall do so 
whenever convenient.
\end{definition}

\noindent It is convenient to write
\[
\cos\theta_\delta = 1-\delta, \qquad \cos\theta_\varepsilon = \frac{\varepsilon}{1-\delta},
\qquad \theta_\delta\in[0,\tfrac{\pi}{2}),\quad \theta_\varepsilon\in(0,\tfrac{\pi}{2}],
\]
so that the condition $\varepsilon<(1-\delta)^2$ is equivalent to $\theta_\varepsilon>\theta_\delta$.

\begin{remark}
The relation $\nperp$ recovers all three classical notions: setting 
$\delta = \varepsilon = 0$ gives exact Birkhoff--James orthogonality; 
$\delta = 0$ gives Chmieli\'nski's $\perp_B^\varepsilon$~\cite{Chm1, ChmStWoj}; 
and $\varepsilon = 0$ gives what we call \emph{$\delta$-norming orthogonality} 
$\perp^0_\delta$, the functional analytic counterpart of Dragomir's approximate 
orthogonality~\cite{Drag}, as made precise by condition~(iv) of 
Theorem~\ref{Thm:Equiv} below. 
\end{remark}

\noindent The condition $\varepsilon < (1-\delta)^2$ ensures non-degeneracy: 
$x \nperp x$ if and only if $x = 0$. The relation is homogeneous and, when 
$x \neq 0$  and $y \neq 0$, forces linear independence of $\{x, y\}$ 
 (this is verified in the course of the proof of 
Theorem~\ref{Thm:Equiv}). It is also monotone: 
$x \nperp y$ implies $x \perp_{\delta'}^{\varepsilon'} y$ for all
$\delta' \geq \delta$, $\varepsilon' \geq \varepsilon$ satisfying
$\varepsilon' < (1-\delta')^2$.

\subsection{Statement of Main Results}

\noindent Our first main result gives a complete characterization of $\nperp$ 
through several equivalent conditions. In the spirit of~\cite{ChmStWoj}, 
Theorem~\ref{Thm:Equiv} extends the functional characterization of 
$\perp_B^\varepsilon$ to the two-parameter setting and to complex spaces.

\begin{theorem}\label{Thm:Equiv}
Let $X$ be a normed linear space, $x, y \in X$ non-zero, and $\delta, \varepsilon 
\in [0,1)$ with $\varepsilon < (1-\delta)^2$. The following conditions are equivalent:
\begin{itemize}
    \item[(i)] $x \nperp y$.
    \item[(ii)] $\exists\, f \in X^*$ with $f(x)$ real and $f(x) \geq (1-\delta)\|f\|\|x\|$ 
    and $|f(y)| \leq \frac{\varepsilon}{1-\delta}\|f\|\|y\|$.
    \item[(iii)] $\exists\, f \in X^*$ with $\mathrm{Re}\,f(x) \geq (1-\delta)\|f\|\|x\|$ 
    and $|f(y)| \leq \frac{\varepsilon}{1-\delta}\|f\|\|y\|$.
    \item[(iv)] $\|x + \lambda y\| \geq (1-\delta)\|x\| - 
    \frac{\varepsilon}{1-\delta}|\lambda|\|y\|$ for all scalars $\lambda$.
\end{itemize}
Moreover, if $X$ is a Hilbert space, then \emph{(i)--(iv)} are further equivalent to:
\begin{itemize}
    \item[(v)] $|\langle x, y\rangle| \leq L(\delta,\varepsilon)\|x\|\|y\|$, where
$L(\delta,\varepsilon)=\cos(\theta_\varepsilon-\theta_\delta)$ is the constant of Proposition~\ref{prop:L}.
   \item[(vi)] $\exists\, u \in S_X$ with $\langle x, u\rangle \geq (1-\delta)\|x\|$ 
    and $|\langle y, u\rangle| \leq \frac{\varepsilon}{1-\delta}\|y\|$.
\end{itemize}
\end{theorem}

\begin{remark}\label{rem:unification}
Condition~(iv) makes the unification explicit. Setting $\varepsilon = 0$ 
gives $\|x + \lambda y\| \geq (1-\delta)\|x\|$ for all scalars $\lambda$, 
which is precisely Dragomir's approximate orthogonality~\cite{Drag}. Setting 
$\delta = 0$ gives $\|x + \lambda y\| \geq \|x\| - \varepsilon|\lambda|\|y\|$, 
the equivalent formulation of Chmieli\'nski's $\perp_B^\varepsilon$ ~\cite{Chm2, Alt} in terms of norm inequality. Setting $\delta = \varepsilon = 0$ recovers 
the Birkhoff--James inequality. Condition~(iv) is thus the natural 
interpolant between the Dragomir and Chmieli\'nski formulations, with $\delta$ 
and $\varepsilon$ independently controlling the two relaxations. In the Hilbert 
space case, the sharp constant $L(\delta,\varepsilon)$ satisfies 
$L(0,\varepsilon) = \varepsilon$ (Chmieli\'nski) and 
$L(\delta,0) = \sqrt{\delta(2-\delta)}$ (Dragomir), interpolating between 
the two boundary cases.
\end{remark}

\noindent The single-vector characterization extends uniformly over closed subspaces, 
yielding a description of $\nperp$ relative to a subspace entirely in terms of 
the restricted norm of a single functional.

\begin{theorem}\label{thm:subspace}
Let $Y$ be a closed subspace of $X$ and $x \in X \setminus Y$. Let $\varepsilon, \delta \in [0,1)$ be such that $\varepsilon < (1-\delta)^2$. Then 
the following are equivalent:
\begin{itemize}
    \item[(a)] $x \nperp y$ 
for every $y \in Y$;
    \item[(b)]  there exists $f \in X^*$ with $\|f\| = 1$ such that
\[
|f(x)| \geq (1-\delta)\|x\|, \qquad \|f|_Y\| \leq \frac{\varepsilon}{1-\delta}.
\]
\end{itemize}
\end{theorem}

\noindent The next result addresses the proximity structure of $\nperp$. The following result proved in \cite{ChmStWoj} serves as an alternative characterization of approximate Birkhoff-James orthogonality.

\begin{theorem*}\cite[Theorem 2.2]{ChmStWoj}
Let $X$ be a real normed space. For $\varepsilon\in [0,1)$ and $x,y\in X$
\[
x\perp^{\varepsilon}_B y \iff \exists~z\in \operatorname{span}\{x,y\}:~ x\perp_B z,~\|z-y\|\leq \varepsilon\|y\|
\]
\end{theorem*}
\noindent This result, later extended in the complex case in \cite{Woj1}. We generalize this to the two-parameter setting.

\begin{theorem}\label{thm:Linspace}
Let $X$ be a normed linear space and $\varepsilon, \delta \in [0,1)$ be such that 
$\varepsilon < (1-\delta)^2$. Then for any $x,y\in X\setminus \{0\}$
\[
x\nperp y \iff \exists~z\in \operatorname{span}\{x,y\}:~x\perp_\delta^0 z,~\text{and}~\|z-y\|\leq \frac{\varepsilon}{1-\delta}\|y\|.
\]
\end{theorem}

\medskip

\noindent The following theorem characterizes norming approximate orthogonality in 
$X^*$ through point evaluations in $X$, with the reflexive case yielding a 
complete equivalence.

\begin{theorem}\label{thm:dual}
Let $X$ be a Banach space, $f, g \in X^*$, and $\delta, \varepsilon \in [0,1)$ with 
$\varepsilon < (1-\delta)^2$. Then $f \perp_\delta^\varepsilon g$ in $X^*$ if there 
exists $x \in S_X$ such that
\[
|f(x)| \geq (1-\delta)\|f\|, \qquad |g(x)| \leq \frac{\varepsilon}{1-\delta}\|g\|.
\]
\end{theorem}

\noindent Building on Theorem~\ref{thm:dual}, we compute the
orthogonality threshold exactly, recovering at $\delta = \varepsilon = 0$ 
the classical characterization of best approximations from a subspace via 
Birkhoff--James orthogonality. Section~\ref{sec:ck} applies the foregoing theory to the space $C(K,X)$ of continuous $X$-valued functions on a compact Hausdorff space $K$. We write $\mathrm{rcabv}(K, X^*)$ to denote the Banach space of regular countably additive 
$X^*$-valued Borel measures of bounded variation on $K$, equipped with the total 
variation norm; we refer to~\cite{CemMen}. Using the Riesz--Markov representation $(C(K,X))^* \cong \mathrm{rcabv}(K,X^*)$ and the Krein--Milman description of the extreme points of the relevant face of the dual ball~\cite{ruess, Lima}, we characterize norming approximate orthogonality in $C(K,X)$ in terms of the $\delta$-norming set. As a 
consequence, through the isometric embedding $\mathcal{K}(X) \hookrightarrow C(B_X^w, X)$ ($B_X^w$: the closed unit ball of a reflexive Banach space $X$ equipped with the 
weak topology), we obtain a characterization of approximate Birkhoff--James 
orthogonality between compact operators entirely in terms of the norm-attainment 
set $M_T = \{x \in  B_X : \|Tx\| = \|T\|\}$. Section \ref{sec5} examines
the geometry of $\nperp$: its strictly two-parameter nature, its symmetry in Hilbert
spaces and its failure of symmetry in general normed linear spaces, and a question on the
rigidity of that symmetry.

\section{Norming Approximate Orthogonality}\label{sec:norming}

\noindent In this Section we prove Theorem \ref{Thm:Equiv}, Theorem \ref{thm:subspace} and Theorem \ref{thm:Linspace}, which mainly deals with equivalent criterion of norming approximate orthogonality between two vectors and between a vector and a subspace. We start with the proof of Theorem \ref{Thm:Equiv}. We begin with a proposition which we need to prove Theorem \ref{Thm:Equiv}.

\begin{proposition}\label{prop:L}
Let $\delta,\varepsilon \in [0,1)$ with $\varepsilon < (1-\delta)^2$, and define
\[
L(\delta,\varepsilon) := \cos(\theta_\varepsilon-\theta_\delta).
\]
Then
\[
L(\delta,\varepsilon) = \varepsilon +
\frac{\sqrt{\delta(2-\delta)\bigl((1-\delta)^2-\varepsilon^2\bigr)}}{1-\delta}
\qquad\text{and}\qquad L(\delta,\varepsilon)<1 .
\]
\end{proposition}

\begin{proof}
Since $\theta_\delta,\theta_\varepsilon\in[0,\frac{\pi}{2}]$, both sines are non-negative
and
\[
\sin\theta_\delta = \sqrt{1-(1-\delta)^2} = \sqrt{\delta(2-\delta)},\qquad
\sin\theta_\varepsilon = \frac{\sqrt{(1-\delta)^2-\varepsilon^2}}{1-\delta}.
\]
Hence
\[
\cos(\theta_\varepsilon-\theta_\delta)
= \cos\theta_\varepsilon\cos\theta_\delta + \sin\theta_\varepsilon\sin\theta_\delta
= \varepsilon + \frac{\sqrt{\delta(2-\delta)\bigl((1-\delta)^2-\varepsilon^2\bigr)}}{1-\delta}.
\]
Moreover $\varepsilon<(1-\delta)^2$ gives $\cos\theta_\varepsilon<\cos\theta_\delta$, so
$\theta_\varepsilon-\theta_\delta\in(0,\frac{\pi}{2}]$ and $L(\delta,\varepsilon)<1$.
\end{proof}

\noindent In the Hilbert space setting, the constant $L(\delta,\varepsilon)$ determines the strength of the orthogonality relation $\nperp$ (see Theorem \ref{Thm:Equiv}). In order for this condition to be non-trivial, it is necessary that $L(\delta,\varepsilon) < 1$. The above observation confirms that this is indeed the case whenever $\varepsilon < (1-\delta)^2$.

\medskip

\noindent Now, we prove Theorem \ref{Thm:Equiv}.

\begin{proof}[Proof of theorem \ref{Thm:Equiv}]
Evidently, $(i)$, $(ii)$, $(iii)$ are mutually equivalent and they are equivalent to $(vi)$, through the Riesz representation Theorem, if additionally, we assume $X$ is a Hilbert space. Therefore, it only remains to prove the equivalence of $(i)$ and $(iv)$ for general normed linear spaces, and equivalence of $(i)$ and $(v)$ under the additional assumption of $X$ is a Hilbert space. Since norming approximate orthogonality is homogeneous, without loss of generality we assume that $\|x\|=\|y\|=1$.

\medskip

\noindent $(i)\implies(iv)$. Suppose there exists $f\in X^*$ 
satisfying~$(i)$. Then for any scalar $\lambda$,
\[
\|f\|\,\|x+\lambda y\|\geq|f(x+\lambda y)|\geq|f(x)|-|\lambda||f(y)|
\geq(1-\delta)\|f\|-\frac{\varepsilon}{1-\delta}|\lambda|\|f\|.
\]
Dividing through by $\|f\|>0$ yields~$(iv)$.

\medskip

\noindent $(iv)\implies(i)$. Define a sublinear functional 
$p:X\to\mathbb{R}$ by
\[
p(v)=\inf_{\lambda\in\mathbb{C}}\left(\|v+\lambda y\|
+\frac{\varepsilon}{1-\delta}|\lambda|\right),\quad v\in X.
\]
Positive homogeneity is immediate. Subadditivity follows from the 
triangle inequality: for any $v_1,v_2\in X$ and $\eta>0$, choosing 
$\lambda_1,\lambda_2$ nearly optimal and setting $\lambda=\lambda_1+\lambda_2$,
\[
p(v_1+v_2)\leq\|v_1+\lambda_1 y\|+\|v_2+\lambda_2 y\|
+\frac{\varepsilon}{1-\delta}(|\lambda_1|+|\lambda_2|)
\leq p(v_1)+p(v_2)+\eta.
\]
Define $h:\mathrm{span}\{x\}\to\mathbb{C}$ by $h(\alpha x)=(1-\delta)\alpha$. 
For any $\alpha\neq 0$ and $\mu\in\mathbb{C}$, condition~$(iv)$ gives
\[
\|x+\mu y\|+\frac{\varepsilon}{1-\delta}|\mu|\geq(1-\delta),
\]
so that $p(\alpha x)\geq|\alpha|(1-\delta)=|h(\alpha x)|$. By the 
Hahn--Banach theorem, $h$ extends to $f\in X^*$ with $|f(v)|\leq p(v)$ 
for all $v\in X$. Taking $\lambda=0$ gives $\|f\|\leq 1$, and taking 
$\lambda=-1$ gives $|f(y)|\leq p(y)\leq\frac{\varepsilon}{1-\delta}$.

\medskip

\noindent Consider the set
\[
\widetilde{N}^\varepsilon_\delta=\left\{g\in X^*:\|g\|\leq 1,\; 
|g(y)|\leq\frac{\varepsilon}{1-\delta}\right\}.
\]
The functional $f$ constructed above belongs to $\widetilde{N}^\varepsilon_\delta$, 
so it is nonempty. It is convex and weak$^*$-closed, hence 
weak$^*$-compact by the Banach--Alaoglu theorem. The map $g\mapsto|g(x)|$ 
is weak$^*$-continuous, so it attains its maximum on 
$\widetilde{N}^\varepsilon_\delta$ at some $f_0\in\widetilde{N}^\varepsilon_\delta$. 
Since $f\in\widetilde{N}^\varepsilon_\delta$,
\[
|f_0(x)|\geq|f(x)|=1-\delta>0.
\]
We shall now prove $\|f_0\|=1$. Since $\varepsilon < (1-\delta)^2$, the 
set $\{x, y\}$ is linearly independent: indeed, if $x = cy$ for some scalar 
$c$, then
\[
\|x-cy\|\geq (1-\delta)\|x\|-\frac{\varepsilon}{1-\delta}\|cy\|=\|x\|\left((1-\delta)-\frac{\varepsilon}{1-\delta}\right)
\]
which gives
\[
(1-\delta)^2\leq \varepsilon,
\]
a contradiction. By the Hahn--Banach Theorem there therefore exists 
$k \in X^*$ with $k(x) \neq 0$, $k(y) = 0$ and $\|k\| = 1$; 
 replacing $k$ by $\overline{\mathrm{sgn}(k(x))}\,k$, we may and do assume $k(x)>0$.
If $\|f_0\|<1$, for sufficiently small $\zeta>0$ with $\|f_0\|+\zeta\leq 1$, define
\[
l= \overline{\mathrm{sgn}(f_0(x))}\,f_0+\zeta k.
\]
Then $\|l\|\leq\|f_0\|+\zeta\leq 1$ and $|l(y)|=|f_0(y)|
\leq\frac{\varepsilon}{1-\delta}$, so $l\in\widetilde{N}^\varepsilon_\delta$. 
However,  $l(x)=|f_0(x)|+\zeta k(x)$ is a sum of two positive reals, 
whence
\[
|l(x)|=|f_0(x)|+\zeta  k(x)>|f_0(x)|,
\]
contradicting the maximality of $f_0$. Therefore $\|f_0\|=1$, and
\[
|f_0(x)|\geq 1-\delta=(1-\delta)\|f_0\|\|x\|,\qquad 
|f_0(y)|\leq\frac{\varepsilon}{1-\delta}=\frac{\varepsilon}{1-\delta}
\|f_0\|\|y\|,
\]
and $f_0$ witnesses condition~$(i)$.

\bigskip

\noindent For the rest of the proof, we now assume that $X$ is a Hilbert space.

\bigskip

\noindent $(iv)\iff(v)$. For any scalar $\lambda=te^{i\theta}$ with $t=|\lambda|\geq 0$,
\[
\|x+\lambda y\|^2=1+2t\,\mathrm{Re}(e^{i\theta}\langle x,y\rangle)+t^2,
\]
which is minimized over all $\lambda$ with $|\lambda|=t$ when $e^{i\theta}\langle x,y\rangle=-|\langle x,y\rangle|$, giving
\[
\min_{|\lambda|=t}\|x+\lambda y\|^2=1-2t|\langle x,y\rangle|+t^2.
\]
Since the right side of $(iv)$ depends only on $t=|\lambda|$, condition $(iv)$ holds for all scalars $\lambda$ if and only if
\[
\forall ~t\geq 0: \qquad \sqrt{1-2t|\langle x,y\rangle|+t^2}\geq(1-\delta)-\frac{\varepsilon}{1-\delta}t.
\]
If the right side is negative the inequality holds trivially, 
so we may assume that $(1-\delta)-\frac{\varepsilon}{1-\delta}t\geq 0$, 
 that is $t\leq T:=\frac{(1-\delta)^2}{\varepsilon}$ (with $T=\infty$ when $\varepsilon=0$) and squaring both sides, condition $(iv)$ is equivalent to
\begin{align}
\forall ~ t\in[0,T]: \qquad 1-2t|\langle x,y\rangle|+t^2\geq\left[(1-\delta)-\frac{\varepsilon}{1-\delta}t\right]^2.\tag{II}
\end{align}
Expanding and rearranging, (II) holds if and only if
\begin{align}
\forall~ t\in[0,T]:\qquad Q(t):=\left(\frac{(1-\delta)^2-\varepsilon^2}
{(1-\delta)^2}\right)t^2+2\left(\varepsilon-|\langle x,y\rangle|
\right)t+\delta(2-\delta)\geq 0.\tag{III}
\end{align}
We set
\[
Q(t) = At^2+2Bt+C, \qquad t\geq 0.
\]
Since $\delta\in[0,1)$ we have $(1-\delta)^2\leq 1-\delta$, and therefore 
$\varepsilon<(1-\delta)^2$ implies $\varepsilon^2<(1-\delta)^2$, so the 
leading coefficient $A=\frac{(1-\delta)^2-\varepsilon^2}{(1-\delta)^2}>0$, 
and the constant term $C=\delta(2-\delta)\geq 0$. Write $B=\varepsilon
-|\langle x,y\rangle|$. We now consider the following cases:

\medskip

\noindent Case I: If $B\geq 0$, that is $|\langle x,y\rangle|
\leq\varepsilon$. Then $Q'(t)=2At+2B\geq 0$ for all $t\geq 0$, so $Q$ 
is increasing on $[0,\infty)$ and $Q(t)\geq Q(0)=C\geq 0$. In this case 
$(III)$ holds automatically and so does $(v)$, since $|\langle x,y
\rangle|\leq\varepsilon\leq L(\delta,\varepsilon)$.

\medskip

\noindent Case II: If $B<0$, that is 
$|\langle x,y\rangle|>\varepsilon$. Since $A>0$, we have $Q(t)\to\infty$ as $t\to\infty$.
Therefore, the continuous convex function $Q$ attains its minimum on
$[0,\infty)$ at the critical point $t^*=-B/A>0$, with minimum value $C-B^2/A$. We first check that this critical point lies in the range $[0,T]$ on which the 
squaring above was legitimate. Since $\|x\|=\|y\|=1$ we have 
$|\langle x,y\rangle|\leq 1$, and therefore
\[
t^*=\frac{\bigl(|\langle x,y\rangle|-\varepsilon\bigr)(1-\delta)^2}
{(1-\delta)^2-\varepsilon^2}
\leq \frac{(1-\varepsilon)(1-\delta)^2}{(1-\delta)^2-\varepsilon^2}
\leq \frac{(1-\delta)^2}{\varepsilon}=T,
\]
the last inequality being equivalent to 
$\varepsilon(1-\varepsilon)\leq (1-\delta)^2-\varepsilon^2$, i.e.\ to 
$\varepsilon\leq (1-\delta)^2$, which holds by hypothesis. Consequently $Q\geq 0$ on 
$[0,T]$ forces $Q\geq 0$ on all of $[0,\infty)$, and (III) is equivalent to the 
corresponding statement on $[0,\infty)$. Therefore 
$Q(t)\geq 0$ for all $t\geq 0$ if and only if
\[
C-\frac{B^2}{A}\geq 0\iff \Delta=4B^2-4AC\leq 0.
\]
Substituting the values of $A$, $B$ and $C$, this becomes
\[
\left(|\langle x,y\rangle|-\varepsilon\right)^2\leq
\frac{\delta(2-\delta)\left((1-\delta)^2-\varepsilon^2\right)}{(1-\delta)^2}.
\]
Since the right side is non-negative, taking square roots gives
\[
|\langle x,y\rangle|\leq\varepsilon+
\frac{\sqrt{\delta(2-\delta)\left((1-\delta)^2-\varepsilon^2\right)}}
{1-\delta},
\]
which is precisely condition $(v)$, by Proposition~\ref{prop:L}.
\bigskip

\noindent This completes the proof.
\end{proof}

\noindent Next we prove Theorem \ref{thm:subspace}.

\begin{proof}[Proof of theorem \ref{thm:subspace}]
The implication  $(b)\Rightarrow(a)$ follows immediately from the
definition. We prove  $(a)\Rightarrow(b)$. By homogeneity, we may assume $\|x\|=1$. Define
\[
p(v)=\inf_{w\in Y}\left(
\|v+w\|
+\frac{\varepsilon}{1-\delta}\|w\|
\right),
\qquad v\in X.
\]
Arguing exactly as in the proof of Theorem~\ref{Thm:Equiv},
$p$ is sublinear. Since $x\nperp y$ for every $y\in Y$, Theorem~\ref{Thm:Equiv}(iv)
implies
\[
\|x+w\|
+
\frac{\varepsilon}{1-\delta}\|w\|
\geq
1-\delta
\qquad \text{for all } w\in Y.
\]
Hence $p(x)\geq 1-\delta$. Define $h:\mathrm{span}\{x\}\to\mathbb{C}$ by
\[
h(\alpha x)=(1-\delta)\alpha.
\]
Then
\[
|h(\alpha x)|
=
|\alpha|(1-\delta)
\leq
p(\alpha x).
\]
By the Hahn--Banach theorem, $h$ extends to
$f\in X^*$ satisfying $|f(v)|\leq p(v)$ for all $v\in X.$ By taking $w=0$, we have $\|f\|\leq 1$. Also, for every $z\in Y$, $|f(z)|\leq p(z)\leq \frac{\varepsilon}{1-\delta}\|z\|,$ and therefore
\[
\|f|_Y\|
\leq
\frac{\varepsilon}{1-\delta}.
\]

\medskip

\noindent Now, consider
\[
\widetilde N_\delta^\varepsilon(Y)
=
\left\{
g\in X^*:
\|g\|\le 1,\ 
\|g|_Y\|
\le
\frac{\varepsilon}{1-\delta}
\right\}.
\]
As in the proof of Theorem~\ref{Thm:Equiv}, the weak$^*$-compactness
of $\widetilde N_\delta^\varepsilon(Y)$ yields a functional
$f_0\in \widetilde N_\delta^\varepsilon(Y)$ maximizing $|g(x)|$.
Since $|f(x)|=1-\delta$, we have
\[
|f_0(x)|\geq 1-\delta.
\]
Using the fact that $x\notin Y$, and mimicking the same argument as in
Theorem~\ref{Thm:Equiv}  (including the modification of the auxiliary 
functional $k$ so that $k(x)>0$) shows that $\|f_0\|=1$. Consequently,
\[
|f_0(x)|
\geq
1-\delta
=
(1-\delta)\|f_0\|\,\|x\|,
\]
and
\[
\|f_0|_Y\|
\leq
\frac{\varepsilon}{1-\delta}.
\]
Thus $f_0$ satisfies  $(b)$. This completes the proof.
\end{proof}

\begin{remark}
Theorem~\ref{thm:subspace} reduces to Theorem~\ref{Thm:Equiv} in the case $Y =
\mathrm{span}\{y\}$, since $\|f|_Y\| \leq \frac{\varepsilon}{1-\delta}$ is then equivalent to
$|f(y)| \leq \frac{\varepsilon}{1-\delta}\|y\|$. 
\end{remark}

\begin{proof}[Proof of Theorem \ref{thm:Linspace}]
We first prove the necessity. Since $x\nperp y$, it follows that $\{x,y\}$ is linearly independent. Let $W$ be the two-dimensional subspace spanned by $x$ and $y$. Observe that $x\nperp y$ in $X$ if and only if $x\nperp y$ in $W$. Therefore, there exists $f\in S_{W^*}$ such that $f(x)\geq (1-\delta)\|x\|$ and $|f(y)|\leq \frac{\varepsilon}{(1-\delta)}\|y\|$. Let $u\in S_W$ be a norm attaining point of $f$ with $f(u)=1$. Choose
\[
z=y-f(y)u,
\]
which lies inside $W$ and satisfies $f(z)=0$. Let $\widetilde{f}$ be a Hahn-Banach extension of $f$. Then $\widetilde{f}(x)=f(x)\geq (1-\delta)\|x\|$ and $\widetilde{f}(z)=f(z)=0$, proving $x\perp_\delta^0 z$. Also, $\|z-y\|=|f(y)|\leq \frac{\varepsilon}{(1-\delta)}\|y\|$.

\medskip

\noindent The sufficiency is straightforward. Since $x\perp_\delta^0 z$, there exists $g\in S_{X^*}$ such that $g(x)\geq (1-\delta)\|x\|$ and $g(z)=0$. It now follows from the hypothesis that $|g(y)|=|g(y)-g(z)|\leq \|y-z\|\leq \frac{\varepsilon}{(1-\delta)}\|y\|.$ This completes the proof.
\end{proof}

\section{Dual formulation of Norming Approximate Orthogonality}\label{sec:dual}

\noindent We start with the proof of Theorem \ref{thm:dual} which gives a sufficient condition for norming approximate orthogonality of continuous linear functionals.

\begin{proof}[Proof of theorem \ref{thm:dual}]
Recall that $f \perp_\delta^\varepsilon g$ in $X^*$ means there exists $\Gamma \in X^{**}$ 
such that
\[
|\Gamma(f)| \geq (1-\delta)\|\Gamma\|\|f\|, \qquad |\Gamma(g)| \leq \frac{\varepsilon}{1-\delta}\|\Gamma\|\|g\|.
\]
Let $\pi: X \to X^{**}$ denote the canonical isometric embedding, defined by 
$\pi(x)(h) = h(x)$ for all $h \in X^*$. Suppose $x \in X$ satisfies the two 
conditions as in the statement of the theorem. Taking $\Gamma = \pi(x)$, we have $\|\Gamma\| = \|x\|$, and
\[
|\Gamma(f)| = |f(x)| \geq (1-\delta)\|f\|\|x\| = (1-\delta)\|\Gamma\|\|f\|,
\]
\[
|\Gamma(g)| = |g(x)| \leq \frac{\varepsilon}{1-\delta}\|g\|\|x\| = \frac{\varepsilon}{1-\delta}\|\Gamma\|\|g\|.
\]
Thus $\pi(x)$ witnesses $f \perp_\delta^\varepsilon g$ in $X^*$, and the proof is complete.
\end{proof}

\noindent In the reflexive case, the converse also holds.

\begin{corollary}\label{cor:reflexdual}
Let $X$ be a reflexive Banach space, $f, g \in X^*$, and $\delta, \varepsilon \in 
[0,1)$ with $\varepsilon < (1-\delta)^2$. Then $f \perp_\delta^\varepsilon g$ in 
$X^*$ if and only if there exists $x \in  S_X$ such that
\[
|f(x)| \geq (1-\delta)\|f\|, \qquad |g(x)| \leq \frac{\varepsilon}{1-\delta}\|g\|.
\]
\end{corollary}

\begin{proof}
The sufficiency follows from Theorem \ref{thm:dual}. For necessity, suppose 
$f \perp_\delta^\varepsilon g$ in $X^*$, so there exists $\Gamma \in X^{**}$ 
witnessing this. Since $X$ is reflexive, the canonical embedding 
$\pi: X \to X^{**}$ is surjective, so $\Gamma = \pi(x)$ for some 
 non-zero $x \in X$, and  normalising $x$ the proof follows.
\end{proof}

\begin{remark}
Therefore, in the reflexive case, $f \perp_\delta^\varepsilon g$ in $X^*$ if and only if there exists $x \in S_X$ such that $f$ approximately attains its norm at $x$, that is, 
$|f(x)| \geq (1-\delta)\|f\|$, and $g$ approximately vanishes at $x$, i.e., 
$|g(x)| \leq \frac{\varepsilon}{1-\delta}\|g\|$.
\end{remark}

\begin{corollary}
Let $X$ be a reflexive Banach space and $f, g \in X^*$. Then $f \perp_B g$ in 
$X^*$ if and only if
\[
M_f \cap \ker g \neq \emptyset,
\]
where $M_f = \{x \in S_X : |f(x)| = \|f\|\}$ is the norm-attainment set of $f$.
\end{corollary}

\begin{proof}
Set $\delta = \varepsilon = 0$ in the corollary above. Then $f \perp_B g$ in 
$X^*$ if and only if there exists $x \in S_X$ such that
\[
|f(x)| = \|f\|, \qquad g(x) = 0,
\]
that is, $x \in M_f \cap \ker g$.
\end{proof}

\noindent We conclude this Section with the proof of Theorem \ref{thm:bound} which recovers  the characterization of best approximation to a point out of a closed subspace.

\begin{proposition}\label{thm:bound}
Let $Y$ be a closed proper non-trivial subspace of $X$, $x \in X \setminus \{0\}$,
and $\delta \in [0,1)$. Put
\[
\widetilde m_\delta(x,Y) := \inf\bigl\{\|f|_Y\| : f \in S_{X^*},\
|f(x)| \geq (1-\delta)\|x\|\bigr\},
\]
\[
m_\delta(x,Y) := \inf\bigl\{\|f|_Y\| : f \in B_{X^*},\
|f(x)| \geq (1-\delta)\|x\|\bigr\},
\]
and suppose $\widetilde m_\delta(x,Y) < 1-\delta$. Then we have
\begin{itemize}
    \item[(a)] $x\notin Y$, both infima are attained.
    \item[(b)] The set
\[
S := \bigl\{\varepsilon \in [0,(1-\delta)^2) : x \perp_\delta^{\varepsilon} Y\bigr\}
\]
has a least element $\varepsilon_{\min}$, and
\[
\varepsilon_{\min} = (1-\delta)\,\widetilde m_\delta(x,Y).
\]
    \item[(c)] $(1-\delta)\,m_\delta(x,Y) \leq \varepsilon_{\min} \leq m_\delta(x,Y)$, and
$x \perp_\delta^{m_\delta(x,Y)} Y$ whenever $m_\delta(x,Y)<(1-\delta)^2$.
\end{itemize}
\end{proposition}

\begin{proof}
Define $\mathscr{N}_\delta(x) := \{ f \in B_{X^*} : |f(x)| \geq (1-\delta)\|x\| \}$.
This set is a weak$^*$-closed subset of $B_{X^*}$, hence weak$^*$-compact by the
Banach--Alaoglu theorem and it is non-empty by the Hahn--Banach
theorem. Since the map $f \mapsto \|f|_Y\|$ is
weak$^*$-lower semicontinuous, it attains its infimum on $\mathscr{N}_\delta(x)$,
so $m_\delta(x,Y)$ is well-defined and the minimum is attained.

\medskip

\noindent The hypothesis forces $x\notin Y$. Indeed, if $x\in Y$, then every $f\in S_{X^*}$ with $|f(x)|\geq(1-\delta)\|x\|$ satisfies $\|f|_Y\|\,\|x\|\geq |f(x)|\geq (1-\delta)\|x\|$, proving $\widetilde m_\delta(x,Y)\geq 1-\delta$, a contradiction.

\medskip

\noindent Let $\varepsilon_* := (1-\delta)\,\widetilde m_\delta(x,Y)$,
so that $\varepsilon_*<(1-\delta)^2$ by hypothesis. Since $x\notin Y$,
Theorem~\ref{thm:subspace} gives, for every $\varepsilon\in[0,(1-\delta)^2)$,
\[
x \perp_\delta^{\varepsilon} Y \iff \exists\, f\in S_{X^*}:~ |f(x)|\geq (1-\delta)\|x\|,~
\|f|_Y\|\leq \tfrac{\varepsilon}{1-\delta}.
\]
If $\varepsilon>\varepsilon_*$ then $\widetilde m_\delta(x,Y)<\varepsilon/(1-\delta)$, so
such an $f$ exists and $\varepsilon\in S$; conversely, if $\varepsilon\in S$ then by Theorem \ref{thm:subspace}, $\widetilde m_\delta(x,Y)\leq \varepsilon/(1-\delta)$, that is
$\varepsilon\geq\varepsilon_*$. Hence $S\neq\emptyset$ and $\inf S=\varepsilon_*$.

\medskip

\noindent To see that the infimum is attained, take any sequence
$\varepsilon_n \searrow \varepsilon_*$
with each $\varepsilon_n \in S$. By the equivalence
$(i)\Leftrightarrow(iv)$ of Theorem~\ref{Thm:Equiv}, for every scalar $\lambda$
and every $y \in Y$,
\[
\|x + \lambda y\|
\geq (1-\delta)\|x\| - \frac{\varepsilon_n}{1-\delta}|\lambda|\|y\|
\xrightarrow{n \to \infty}
(1-\delta)\|x\| - \frac{\varepsilon_*}{1-\delta}|\lambda|\|y\|,
\]
so condition $(iv)$ holds at $\varepsilon_*$ for every $y \in Y$, and by
$(iv)\Rightarrow(i)$ of Theorem~\ref{Thm:Equiv} we conclude
$x \perp_\delta^{\varepsilon_*} Y$. Therefore, $\varepsilon_{\min}=\varepsilon_*=(1-\delta)\widetilde m_\delta(x,Y)$, and there exists $h\in S_{X^*}$ with
$|h(x)|\geq(1-\delta)\|x\|$ and $\|h|_Y\|\leq \varepsilon_{\min}/(1-\delta)=\widetilde m_\delta(x,Y)$. This also shows
\[
\widetilde m_\delta(x,Y) := \min\bigl\{\|f|_Y\| : f \in S_{X^*},\
|f(x)| \geq (1-\delta)\|x\|\bigr\},
\]

\medskip

\noindent Evidently, $m_\delta(x,Y)\leq \widetilde m_\delta(x,Y)$, since $S_{X^*}\subset B_{X^*}$. Conversely, let $f_0 \in \mathscr{N}_\delta(x)$ attain $m_\delta(x,Y)$. From
$\|f_0\|\,\|x\|\geq|f_0(x)| \geq (1-\delta)\|x\|$ we get $\|f_0\| \geq 1-\delta > 0$, so
$g_0 := f_0/\|f_0\| \in S_{X^*}$ satisfies
\[
|g_0(x)| = \frac{|f_0(x)|}{\|f_0\|} \geq (1-\delta)\|x\|,\qquad
\|g_0|_Y\| = \frac{\|f_0|_Y\|}{\|f_0\|}\leq \frac{m_\delta(x,Y)}{1-\delta},
\]
proving $\widetilde m_\delta(x,Y)\leq m_\delta(x,Y)/(1-\delta)$. Multiplying the chain
$m_\delta(x,Y)\leq \widetilde m_\delta(x,Y)\leq m_\delta(x,Y)/(1-\delta)$ by $(1-\delta)$
gives
\[
(1-\delta)\,m_\delta(x,Y) \leq \varepsilon_{\min} \leq m_\delta(x,Y).
\]
If in addition $m_\delta(x,Y)<(1-\delta)^2$, then $m_\delta(x,Y)\in S$ by monotonicity, so
$x \perp_\delta^{m_\delta(x,Y)} Y$. This completes the proof.
\end{proof}

\begin{corollary}
Let $X$ be a normed linear space, $Y$ a closed proper non-trivial subspace 
of $X$, and $x \in X \setminus Y$. Then $y_0 \in Y$ is a best approximation 
to $x$ from $Y$ if and only if
\[
m_0(x-y_0, Y) := \min\left\{\|f|_Y\| : f \in S_{X^*},\ 
f(x - y_0) = \|x - y_0\|\right\} = 0.
\]
\end{corollary}

\begin{proof}
By the classical characterization of best approximations, $y_0 \in Y$ is 
a best approximation to $x$ from $Y$ if and only if
\[
x - y_0 \perp_B y \quad \text{for every } y \in Y.
\]
Setting $\delta = \varepsilon = 0$ in Theorem~\ref{thm:subspace} applied 
to $x - y_0 \in X \setminus Y$, this holds if and only if there exists 
$f \in S_{X^*}$ such that $f(x-y_0) = \|x-y_0\|$ and $f|_Y \equiv 0$, 
which is precisely the condition $m_0(x-y_0, Y) = 0$.
\end{proof}

\section{Application to continuous function spaces}\label{sec:ck}

\noindent In this Section we provide some applications of the theory of norming approximate orthogonality developed so far, in case of continuous function spaces. Given any continuous linear functional $f$ on a normed linear space $X$, we consider the $\delta$-norming set and absolute $\delta$-norming set of $f$ as 
\[
M_\delta(f) = \{x\in S_X:~|f(x)|\geq (1-\delta)\|f\|\},~A_\delta(f) = \{x\in S_X:~f(x)\geq (1-\delta)\|f\|\}
\]
for  $\delta\in[0,1)$. We note that $M_0(f)$ is the exact norm attainment set of $f$, which can be empty in the non-reflexive case, however, $M_\delta(f)\neq \emptyset$ in any normed linear space provided $\delta > 0$. We have the following characterization of approximate norming orthogonality in terms of approximate norm attainment set which is a reformulation of Definition \ref{1st Defn}. 

\begin{proposition}\label{prop:intersection}
Let $X$ be a normed linear space and let $\delta,\varepsilon\in [0,1)$ with $\varepsilon < (1-\delta)^2$.  Let $x,y\in X$ be non-zero. Then $x\nperp y$ if and only if
\begin{align}\label{Intersection}
\left\{f(y):~f\in S_{X^*},~\frac{x}{\|x\|}\in A_\delta(f)\right\}\cap B\left(0,\frac{\varepsilon}{(1-\delta)}\|y\|\right) \neq \emptyset.
\end{align}
\end{proposition}

\noindent The above reformulation, although elementary, has deeper implications in case of $\delta=0$ through the Krein-Milman Theorem. However, first we formulate norming approximate orthogonality in the continuous function space $C(K,X)$ endowed with the supremum norm, where $K$ is a compact Hausdorff space and $X$ is a normed linear space. For any $\phi\in C(K,X)$, we define the norm attainment of $\phi$ as
\[
M_\phi = \{t\in K:~\|\phi(t)\|=\|\phi\|_\infty\}
\]
which is always non-empty, since $K$ is compact. The dual $(C(K,X))^*$ of $C(K,X)$ can be identified isometrically with $rcabv(K,X^*)$, the space of regular countably additive $X^*$-valued Borel measures 
of bounded variation on $K$ with the dual pairing:
\[
m(\phi) = \int_K \phi \, dm, \quad \phi \in C(K,X),\; m \in rcabv(K,X^*),
\]
with $\|m\| = |m|(K)$. Thus, Proposition \ref{prop:intersection} takes the following from in $C(K,X)$:

\begin{proposition}\label{Prop:Intersection}
Let $\phi, \psi \in C(K,X)$  be non-zero and $\delta,\varepsilon\in [0,1)$ with $\varepsilon < (1-\delta)^2$. Then $\phi\nperp \psi$ if and only if
\[
\left\{ \int_K \psi \, dm:~  m\in S_{(C(K,X))^*},~ \int_K \phi \, dm\geq (1-\delta)\|\phi\|_\infty\right\} \cap B\left(0,\frac{\varepsilon}{(1-\delta)}\|\psi\|\right) \neq \emptyset.
\]
\end{proposition}

\medskip

\noindent For $t \in K$, we consider the Dirac measure $\delta_t$ on $\mathcal{B}(K)$, the $\sigma$-algebra of all Borel subsets of $K$, by
\[
\delta_t(A)=\begin{cases}
1, & t\in A\\
0, & t\notin A.
\end{cases}
\]
For any $t\in K$ and $x^* \in X^*$, we consider the measure $\delta_t\otimes x^*$ as a linear functional on $C(K,X)$, defined by
\[
x^* \otimes \delta_t : C(K,X) \to \mathbb{C}, 
\quad (x^* \otimes \delta_t)(\phi) = x^*(\phi(t)).
\]
Then for every $\phi \in C(K,X)$,
\[
\int_K \phi \, d(x^* \otimes \delta_t) = x^*(\phi(t)).
\]

\medskip

\noindent Let $\phi\in C(K,X)$ with $\|\phi\|_\infty=1$. Consider the collection of norm one measures
\[
F_0=\{m\in S_{(C(K,X))^*}:~m(\phi)=\int_K \phi \, dm=1\}.
\]
The collection $F_0$ is a face of $B_{(C(K,X))^*}$ and is weak$^*$-closed. Therefore, by the Krein--Milman theorem, $F_0$ is the weak$^*$-closed convex hull of its extreme points. It follows from \cite{Bro, ruess} that 
\[
\mathrm{ext}(B_{(C(K,X))^*})=\{x^*\otimes \delta_t:~x^*\in \mathrm{ext}(B_{X^*}),~t\in K\}.
\]
Since $\mathrm{ext}(F_0)\subset \mathrm{ext}(B_{(C(K,X))^*})$, and $x^*\otimes \delta_t\in F_0$ if and only if $x^*(\phi(t))=\|\phi\|_\infty$, which together with $\|x^*\|=1$ forces $t\in M_\phi$, we have
\begin{align*}
F_0 
&= \overline{\mathrm{conv}}^{w^*}\{x^*\otimes \delta_t:~x^*\in \mathrm{ext}(B_{X^*}),~t\in K,~x^*(\phi(t))=\|\phi\|_\infty\}\\
&= \overline{\mathrm{conv}}^{w^*}\{x^*\otimes \delta_t:~x^*\in \mathrm{ext}(B_{X^*}),~t\in M_\phi,~x^*(\phi(t))=\|\phi\|_\infty\}\\
&= \overline{\mathrm{conv}}^{w^*}\{x^*\otimes \delta_t:~x^*\in S_{X^*},~t\in M_\phi,~x^*(\phi(t))=\|\phi\|_\infty\}.
\end{align*}
The first two equalities follow from the above observation. For the third equality, the inclusion $\subset$ is clear. For $\supset$, fix $t\in M_\phi$ and $x^*\in S_{X^*}$ with $x^*(\phi(t))=\|\phi\|_\infty$. Then $x^*$ belongs to the weak$^*$-closed face $\{y^*\in B_{X^*}:~y^*(\phi(t))=\|\phi\|_\infty\}$ of $B_{X^*}$. Hence, by the Krein--Milman theorem,
\[
x^*\in\overline{\mathrm{conv}}^{w^*}\{y^*\in \mathrm{ext}(B_{X^*}):~y^*(\phi(t))=\|\phi\|_\infty\}.
\]
Applying the weak$^*$-continuous linear map $y^*\mapsto y^*\otimes \delta_t$, we obtain
\[
x^*\otimes \delta_t\in\overline{\mathrm{conv}}^{w^*}\{y^*\otimes \delta_t:~y^*\in \mathrm{ext}(B_{X^*}),~y^*(\phi(t))=\|\phi\|_\infty\}.
\]
Taking weak$^*$-closed convex hulls over all such $t$ and $x^*$ gives the reverse inclusion. Therefore, with $\delta=0$, we get the following result, as a corollary to Proposition \ref{Prop:Intersection}.

\begin{corollary}\label{Cor:CKX}
Let $\phi, \psi \in C(K,X)$ and $\varepsilon\in  [0,1)$. Then $\phi\perp^{\varepsilon}_0 \psi$ if and only if
\[
\mathrm{conv}\{x^*(\psi(t)):\ t\in M_\phi,\ x^*\in S_{X^*},\ x^*(\phi(t))=\|\phi\|_\infty\}
\cap B(0,\varepsilon\|\psi\|)\neq\emptyset.
\]
\end{corollary}

\begin{proof}
In the case $\delta=0$, we have
\[
F_0=\overline{\mathrm{conv}}^{w^*}\{x^*\otimes \delta_t:\ t\in M_\phi,\ x^*\in S_{X^*},\ x^*(\phi(t))=\|\phi\|_\infty\}.
\]
Applying the weak$^*$-continuous affine map $m\mapsto \int_K \psi\,dm$, it follows that
\[
\left\{\int_K \psi\,dm:\ m\in F_0\right\}
=
\overline{\mathrm{conv}}\{x^*(\psi(t)):\ t\in M_\phi,\ x^*\in S_{X^*},\ x^*(\phi(t))=\|\phi\|_\infty\}.
\]
To remove the closure, consider the set
\[
E=\{(t,x^*)\in K\times B_{X^*}:\ x^*(\phi(t))=\|\phi\|_\infty\}.
\]
Since $K$ is compact and $B_{X^*}$ is weak$^*$-compact, the product $K\times B_{X^*}$ is compact. Moreover, the map $(t,x^*)\mapsto x^*(\phi(t))$ is continuous (norm in $t$, weak$^*$ in $x^*$), so $E$ is closed and hence compact. The map $(t,x^*)\mapsto x^*(\psi(t))$ is also continuous, and therefore the set
\[
\{x^*(\psi(t)):\ t\in M_\phi,\ x^*\in S_{X^*},\ x^*(\phi(t))=\|\phi\|_\infty\}
\]
is a compact subset of $\mathbb{C}$. Consequently, by Carath\'{e}odory's Theorem on convex hull, its convex hull is also a compact subset of $\mathbb{C}$, and in particular closed, so the closure may be omitted. The conclusion now follows from  Proposition~\ref{Prop:Intersection}.
\end{proof}
\noindent The above corollary gives a shorter proof of \cite[Theorem 3.2]{roy1} (for $\varepsilon=0$) and an independent proof of \cite[Theorem 4.4]{roy1} and extends them by removing the finite dimensionality of $X$. Moreover, particular case of this corollary extends \cite[Theorem 4.4]{roy1} for compact operators, as we illustrate below. Let $X$ be a reflexive Banach space. Then $B_X$ equipped with the weak topology is a compact Hausdorff space, which we denote $B_X^w$. Since every member of $K(X)$, the space of all compact operators are weak to strong continuous, we may identify $K(X)$ inside $C(B_X^w, X)$ through the isometric embedding
\[
K(X)\hookrightarrow C(B_X^w,X),\qquad T\mapsto T|_{B_X},
\]
since $\|T\|=\sup_{x\in B_X}\|Tx\|=\|T|_{B_X}\|_\infty$. For $T\in K(X)$, define
\[
M_T=\{x\in B_X:\|Tx\|=\|T\|\}.
\]
Since $x\mapsto \|Tx\|$ is weakly continuous (by complete continuity) and $B_X^w$ 
is compact, the supremum is attained, so $M_T\neq\emptyset$.

\begin{corollary}
Let $X$ be a reflexive Banach space, $T,A\in K(X)$, and $\varepsilon\in[0,1)$. 
Then $T\zperp A$ if and only if
\[
\operatorname{conv}\{y^*(Ax):x\in M_T,\;y^*\in J(Tx)\}
\cap B(0,\varepsilon\|A\|)\neq\emptyset.
\]
\end{corollary}

\begin{proof}
Apply Corollary~\ref{Cor:CKX} with $K=B_X^w$, $\varphi=T|_{B_X}$, and $\psi=A|_{B_X}$. Since condition (iv) of Theorem~\ref{Thm:Equiv} depends only on the
quantities $\|T+\lambda A\|$, the relation $\zperp$ is unchanged when $K(X)$ is regarded
as a subspace of $C(B_X^w,X)$. Under the isometric embedding $K(X)\hookrightarrow C(B_X^w,X)$, we have 
$\|\varphi\|_\infty=\|T\|$ and
\[
M_0(\varphi)=\{x\in B_X:\|Tx\|=\|T\|\}=M_T,\qquad J(\varphi(x))=J(Tx),
\]
so the conclusion follows immediately from Corollary~\ref{Cor:CKX}.
\end{proof}

\section{Symmetry and concluding remarks}\label{sec5}

\noindent In this section we examine geometric aspects of the proposed orthogonality, including its lack of symmetry in general Banach spaces and its symmetric behaviour in Hilbert spaces. We would like to mark that $\nperp$ is strictly a two-parameter orthogonality, as illustrated in the following example.

\noindent We first record the elementary comparison with Chmieli\'nski's relation.

\begin{proposition}\label{prop:compare}
Let $\delta\in[0,1)$ and $\varepsilon\in[0,(1-\delta)^2)$, and let $x,y\in X$. Then $x\zperp y$ implies 
$x\nperp y$.
\end{proposition}

\begin{proof}
Let $f\in S_{X^*}$ satisfy $f(x)=\|x\|$ and $|f(y)|\leq\varepsilon\|y\|$. Then 
$|f(x)|\geq (1-\delta)\|x\|$, and since $\delta\geq 0$ we have 
$\varepsilon\leq\frac{\varepsilon}{1-\delta}$, so 
$|f(y)|\leq\frac{\varepsilon}{1-\delta}\|y\|$.
\end{proof}

\noindent The following example shows that the implication in 
Proposition~\ref{prop:compare} is strict for $\delta>0$, so that $\nperp$ is strictly 
weaker than $\zperp$.

\begin{example}\label{ex:2-para}
Consider the two-dimensional real Banach space $X=\ell_\infty^2$. Consider the unit vectors
\[
x_0=(1,0.6),~y_0=(1,0.4).
\]
Let $\varepsilon=0.2$ and $\delta=0.5$. Then $\frac{\varepsilon}{1-\delta}=0.4.$ Consider the norm one continuous linear functional $f:X\to \mathbb{R}$ by 
\[
f(u,v)=v,\qquad (u,v)\in \ell_\infty^2.
\]
Then 
\[
f(x_0)=0.6\geq (1-\delta)=0.5,~f(y_0)=0.4\leq \frac{\varepsilon}{1-\delta}=0.4.
\]
Therefore, $x_0\nperp y_0$. However, it is easy to see that $x_0$ is a smooth point in $X$ and therefore, has a unique support functional given by
\[
g_0(u,v)=u, \qquad (u,v)\in \ell_\infty^2.
\]
However, since $g_0$ also supports $y_0$, $g_0(y_0)=1$, which shows $x_0\not\perp^{\gamma}_B y_0$ for any $\gamma\in [0,1)$, i.e., $x_0$ is not approximately orthogonal to $y_0$ in the sense of Chmieli\'nski.
\end{example}

\begin{remark}
It is well known that Birkhoff--James orthogonality is symmetric in a normed linear space $X$ with $\dim X \geq 3$ if and only if $X$ is an inner product space. This naturally raises the question of whether symmetry of the two-parameter orthogonality relation $\nperp$ imposes similar geometric restrictions on the underlying space. However, symmetry of $\nperp$ (for fixed $\delta, \varepsilon > 0$) does not a priori imply symmetry of Birkhoff--James orthogonality, since the latter corresponds to the limiting case $(\delta,\varepsilon) = (0,0)$, where the defining conditions become exact rather than approximate. Thus, the behaviour of $\nperp$ may differ substantially from that of $\perp_B$.
\end{remark} 

\noindent The following example shows that $\nperp$ is not symmetric in general.

\begin{example}\label{ex:nonsym}
We work with the same space $X$, $\varepsilon$, $\delta$ as in Example~\ref{ex:2-para}. We know that $x_0 \perp_{\delta}^{\varepsilon} y_0$. We show that $y_0 \not\perp_{\delta}^{\varepsilon} x_0$. The dual of $X$ is $\ell_1^2$ and the dual action is point-wise multiplication. Every norm-one functional $g$ can be identified with a unique vector tuple $(a,b)$ with $|a|+|b|=1$. 

\medskip

\noindent\textit{Case I}: $a \geq 0,\; b \geq 0$, so $a + b = 1$.
The condition $g(y_0) = a + 0.4b \geq 0.5$ gives, on substituting 
$a = 1-b$
\[
1 - 0.6b \geq 0.5 \implies b \leq \tfrac{5}{6}.
\]
Then
\[
g(x_0) = (1-b) + 0.6b = 1 - 0.4b 
\;\geq\; 1 - 0.4 \cdot \tfrac{5}{6} = \tfrac{2}{3} > 0.4.
\]

\medskip

\noindent\textit{Case~II:} $a \geq 0,\; b \leq 0$, so $a - b = 1$, 
i.e., $a = 1+b$ with $b \in [-1,0]$.
The condition $g(y_0) = (1+b) + 0.4b = 1 + 1.4b \geq 0.5$ gives 
$b \geq -\frac{5}{14}$, so $b \in [-\frac{5}{14}, 0]$.
Then $g(x_0) = (1+b) + 0.6b = 1 + 1.6b$ is increasing in $b$, 
with minimum
\[
g(x_0) \;\geq\; 1 + 1.6\cdot\bigl(-\tfrac{5}{14}\bigr) 
= 1 - \tfrac{4}{7} = \tfrac{3}{7} > 0.4.
\]

\medskip
\noindent\textit{Case~III:} $a \leq 0,\; b \geq 0$, so $b - a = 1$, 
i.e., $a = b - 1$ with $b \in [0,1]$.
The condition $g(y_0) = (b-1) + 0.4b = 1.4b - 1 \geq 0.5$ requires 
$b \geq \frac{15}{14} > 1$, which is impossible and therefore, such a case does not arise.

\medskip

\noindent\textit{Case~IV:} $a \leq 0,\; b \leq 0$.
Then $g(y_0) = a + 0.4b \leq 0 < 0.5$, and such a case is also not feasible.

\medskip

\noindent Therefore, we conclude that
\[
|g(x_0)| \geq \frac{3}{7} > 0.4, \quad \forall g\in S_{X^*}:~g(y_0)\geq 0.5.
\]
Consequently,
\[
y_0\not\nperp x_0.
\]
Thus, norming approximate orthogonality is not symmetric in general.
\end{example}

\begin{remark}
In contrast to the general normed space setting, the situation is markedly different in Hilbert spaces. In this case, by Theorem~\ref{Thm:Equiv}, the relation $\nperp$ admits a characterization in terms of the inner product:
\[
|\langle x, y\rangle| \leq L(\delta,\varepsilon)\|x\|\|y\|,
\]
where $\varepsilon < (1-\delta)^2$. This condition is clearly symmetric in $x$ and $y$, and therefore $\nperp$ is symmetric in Hilbert spaces.
\end{remark}

\noindent On the other hand, as seen in Example~\ref{ex:nonsym}, $\nperp$ fails to be symmetric in general normed spaces. In view of this, we close this paper with the following open question.

\bigskip

\begin{question*}
Let $X$ be a Banach space with three or more dimensions. For a fixed pair $\delta,\varepsilon\in(0,1)$ with $\varepsilon<(1-\delta)^2$, suppose the norming approximate orthogonality relation $\nperp$ is symmetric. Does it follow that
$X$ is a Hilbert space?
\end{question*}









\end{document}